\documentclass[a4paper,11pt]{article}
\usepackage{authblk}
\usepackage[english]{babel} 
\usepackage[centertags,intlimits]{amsmath}
\usepackage{amsfonts}
\usepackage{amsthm}
\usepackage{amssymb}
\usepackage{pgf, tikz, pgfplots, float, graphicx, tikz-cd}
\usetikzlibrary{decorations.pathreplacing,calligraphy, arrows}

\newtheorem{theorem}{Theorem}[section]
\newtheorem{proposition}[theorem]%
{Proposition}
\newtheorem{definition}{Definition}[section]
\newtheorem{lemma}[theorem]%
{Lemma}
\newtheorem{corollary}[theorem]%
{Corollary}

\newtheorem{example}{Example}[section]

\begin{document}

\title{Remoteness of graphs with given size and connectivity constraints}

\author[1]{Peter Dankelmann\thanks{Financial support by the South Africa National Research Foundation is greatly acknowledged.}}
\author[1,2]{Sonwabile Mafunda}
\author[1]{Sufiyan Mallu\thanks{The results of this paper form part of this authors PhD thesis.}
\thanks{Financial support by the South Africa National Research Foundation is greatly acknowledged.}}
\affil[1]{University of Johannesburg\\
South Africa}
\affil[2]{Soka University of America\\
USA}


\maketitle

\begin{abstract}
Let $G$ be a finite, simple connected graph. The average distance of a vertex $v$ of $G$ is the arithmetic mean of the distances 
from $v$ to all other vertices of $G$. The remoteness $\rho(G)$ of $G$ 
is the maximum of the average distances of the vertices of $G$.

In this paper, we give sharp upper bounds on the remoteness of a graph of given order, 
connectivity and size. We also obtain corresponding bound s for $2$-edge-connected 
and $3$-edge-connected graphs, and bounds in terms of order and size for triangle-free graphs.
\end{abstract}

MSC-class: 05C12

\section{Introduction}
In this paper we consider finite, connected graphs with no loops or multiple edges. In a graph $G$ of order at least two, the average distance $\bar{\sigma}_G(v)$ of a vertex is defined to be the arithmetic mean of the distances from $v$ to all other vertices of $G$, i.e., $\overline{\sigma}_G(v) = \frac{1}{n-1} \sum_{w \in V} d_G(v,w)$, where the distance $d_G(v,w)$ denotes the usual shortest path distance. The remoteness $\rho(G)$ of a connected graph $G$ is defined as the maximum of the average distances of the vertices of $G$,
i.e. 
\[\rho(G)={\rm max}_{v\in V}\bar{\sigma}_G(v)\] where $V$ is the vertex set of $G$.\\
The term remoteness was first used in a paper on automated comparison of graph invariants \cite{AouCapHan2007}, and is in wide use nowadays. However, the remoteness of graphs and closely related concepts had been studied before under different names.  
Zelinka  [29] studied the {\em vertex deviation}, defined as $\frac{\sigma_G(v)}{n}$, 
where $\sigma_G(v)$ denotes the sum of the distances between $v$ and all other vertices,
and $n$ is the number of vertices. 
Other authors used terms such as {\em transmission}, for example, \cite{Ple1984}, {\em total distance} or simply {\em distance} for $\sigma_G(v)$ of $v$.\\ 

Bounds on remoteness in terms of order only were given by Zelinka \cite{Zel1968} and later, independently, by Auochiche and Hansen \cite{AouHan2011}. 

\begin{theorem}\label{Zel1968AouHan2011}
{\rm (Zelinka \cite{Zel1968}, Aouchiche, Hansen \cite{AouHan2011})} \\
Let $G$ be a connected graph of order $n\geq 2$. Then
\[ \rho(G) \leq \frac{n}{2}, \]
with equality if and only if $G$ is a path.
\end{theorem}

In \cite{EntJacSny1976}, Entringer, Jackson and Snyder considered results by Ore \cite{Ore1968} to strengthen the results in Theorem \ref{Zel1968AouHan2011} by taking into account a bound on the size of a graph. 

\begin{theorem}\label{EntJacSny1976}
{\rm (Entringer, Jackson, Snyder \cite{EntJacSny1976})}\\
Let $G$ be a connected graph of order $n$ and size at least $m$. Then
\[\rho(G) \; \leq \;\dfrac{n+2}{2}-\dfrac{m}{n-1}.\]
\end{theorem}

The proof of Theorem \ref{EntJacSny1976} given in \cite{EntJacSny1976} is elegant and short,
but it neither yields the extremal graphs, nor does it generalise in a natural way, 
for example to
graphs of higher connectivity. Using an entirely different approach, we show that the bound in 
Theorem \ref{EntJacSny1976} can be strengthened considerably for $\kappa$-connected 
graphs, where $\kappa$ is arbitrary and for $\lambda$-edge-connected graphs, 
where $\lambda \in \{2,3\}$.  We also characterise all 
extremal graphs for $\kappa=1$. Using a proof similar to that of Theorem \ref{EntJacSny1976} 
we also obtain an improved bound for triangle-free graphs, which we show to be sharp.

The literature contains several results on remoteness of graphs, ranging from bounds 
on remoteness of different classes of graphs to the relation between remoteness and other 
graph parameters. There are results on remoteness in maximal planar graphs and maximal bipartite planar graphs  \cite{CzaDanOlsSze2021}, in graphs not containing certain cycles as 
\cite{DanJonMaf2021}, and trees \cite{BarEntSze1997, Zel1968}. Also relations between 
remoteness and other graph parameters have been studied, for example girth \cite{AouHan2017},  
minimum degree see \cite{Dan2015}, maximum degree see \cite{DanMafMal2022} and clique 
number see \cite{HuaDas2014}. A survey on proximity and remoteness in graphs has been given in \cite{Aou2024}. Recently, 
the bound in Theorem \ref{Zel1968AouHan2011} was extended to digraphs by Ai, Gerke, Gutin 
and Mafunda \cite{AiGerGutMaf2021}.

\begin{theorem}\label{AiGerGutMaf2021}
{\rm (Ai, Gerke, Gutin, Mafunda \cite{AiGerGutMaf2021})}\\
Let $D$ be a strong digraph of order $n\geq 3$. Then 
\[\rho(D)\leq \frac{n}{2},\]
with equality if and only if $D$ is strong and contains a Hamiltonian dipath 
$v_1v_2\dots v_n$ such that no directed edge of the form
$v_iv_j$ with $2\le i+1<j\le n$ is in $D$.
\end{theorem}

\section{Terminology and notation}
We use the following notation. 
For a graph $G$ we denote by $V(G)$ and $E(G)$ the {\em vertex set} and {\em edge set}, respectively.  
The {\em order} and {\em size} of a graph are  the cardinalities of the vertex set and edges, respectively. 
By an $(n,m)$-graph we mean a graph of order $n$ and size at least $m$.
If $G$ and $G'$ are distinct graphs with the same vertex set, but $E(G) \subset E(G')$, i.e., $G'$ is obtained from $G$ by adding edges, then we write $G \lneqq G'$.

If $v$ is a vertex in a graph $G$, then the {\em neighbourhood} $N(v)$ of $v$ is the set of all 
vertices adjacent to $v$.
We will say a graph $G$ is {\em $\kappa$-connected} or {\em $\lambda$-edge-connected} if removal 
of fewer than $\kappa$ vertices or fewer than $\lambda$ edges, respectively, 
leaves $G$ connected.

The {\em eccentricity} ${\rm ecc}_G(v)$ of a vertex $v$ in a graph $G$ is the distance 
from $v$ to a vertex farthest from $v$. The largest of all
eccentricities of vertices of $G$ is called the {\em diameter} and is denoted by 
${\rm diam}(G)$. For $i\in \mathbb{Z}$ let $N_i(v)$ be 
the set of all vertices $u$ with $d(v,u)=i$, and let $n_i(v)=|N_i(v)|$. 
Clearly, $n_i(v) >0$ if and only if $0\leq i \leq {\rm ecc}_G(v)$. 

We denote the complete graph of order $n$ by $K_n$. 
If $G_1, G_2,\ldots,G_k$ are disjoint graphs, then the {\em sequential
sum} $G_1+G_2+\cdots G_k$ is the graph obtained from the disjoint
union of $G_1, G_2,\ldots, G_k$ by joining every vertex in $G_i$ to
every vertex in $G_{i+1}$ for $i=1,2,\ldots,k-1$.
If $t, k\in \mathbb{N}$, then $[K_{a_1} + K_{a_2}+\ldots+K_{a_t}]^k$ 
stands for $k$ repetitions of the pattern 
$K_{a_1} + K_{a_2}+\ldots+K_{a_t}$.
 If $G$ is a graph, then the
{\em complement} of $G$, denoted by $\overline{G}$, is the graph on
the same vertex set as $G$, in which two vertices are adjacent if
they are not adjacent in $G$. 
We say that a graph $G$ is {\em triangle-free} if $G$ contains no subgraph 
isomorphic to $K_3$. 

If $(k,b)$ and $(k',b')$ are distinct pairs of integers, then we write $(k,b) \preceq (k',b')$ if 
$(k,b)$ comes before $(k',b')$ in the lexicographic ordering of pairs of integers, i.e., if
either $k<k'$ or $k=k'$ and $b< b'$. For triples of integers the relation $\preceq$ is defined
analogously.

\section{Maximum remoteness of a $\kappa$-connected graph with given order and size}\label{vertex-connectivity}

In this section we give a sharp upper bound on the remoteness of a $\kappa$-connected graph in terms of order and size. Our results generalize the bound on remoteness in connected graphs
of given order and size in Theorem \ref{EntJacSny1976}. Our proof is entirely different 
from that in \cite{EntJacSny1976}, this technique allows us to determine the extremal graphs in Theorem \ref{EntJacSny1976}, i.e., if $\kappa=1$, and also for many values of the 
size if $\kappa>1$.

We first describe a family of $\kappa$-connected graphs of order $n$ and size at least $m$, which we term $\kappa$-connected path-complete graphs. We prove some properties of
these graphs.   
We then show that path-complete graphs attain the maximum remoteness among all $\kappa$-connected graphs of order $n$ and size at least $m$. 

\begin{definition} \label{def:kappa-connected-path-complete}
A graph $G$ is said to be a $\kappa$-connected path-complete graph if there exist
$\ell, a, b \in \mathbb{N}$, $a \geq \kappa$, with  
\[ G = K_1 + [K_{\kappa}]^{\ell} + K_a + K_b.\]
\end{definition}

For graphs of diameter greater than $2$, Definition \ref{def:kappa-connected-path-complete}
generalises the path-complete graphs defined by Solt\'{e}s in \cite{Sol1991}, which 
are the $1$-connected path-complete graphs defined above.

\begin{lemma}   \label{la:size-vs-remoteness-in-kappa-connected-pc}
(a) Let $H$ be a $\kappa$-connected path-complete graph, where $\kappa \in \mathbb{N}$. Then 
$\rho(H+e) < \rho(H)$ for every edge $e \in E(\overline{H})$. \\
(b) Let $H, H'$ be two distinct $\kappa$-connected path-complete graphs of order $n$. 
Then either $m(H) < m(H')$ and $\rho(H) > \rho(H')$, or $m(H) > m(H')$ 
and $\rho(H) < \rho(H')$. \\
(c) Given $n, \kappa$. Then there exists a $\kappa$-connected path-complete graph of order 
$n$ and size $m$ if an only if 
$m \equiv \binom{n-1}{2} \pmod{\kappa}$ and
$\frac{1}{2} [n(3\kappa-1) - 2\kappa^2 - \kappa+1 - b(\kappa-b)] 
\leq m \leq \binom{n-1}{2}$, where $b$ is the integer in $\{1,2,\ldots, \kappa\}$ with
$b \equiv n-1 \pmod{\kappa}$.
\end{lemma}

\begin{proof}
(a) There exist $a,b,\kappa \in \mathbb{N}$ with $H = K_1 + [K_{\kappa}]^{\ell} + K_a + K_b$.
If $(a,b) \neq (\kappa,1)$, then the vertex in $K_1$ is the unique vertex attaining 
the remoteness. If $(a,b)=(\kappa,1)$, then the vertex in $K_1$ and the vertex in $K_b$
are the unique vertex attaining the remoteness. In both cases, adding an arbitrary edge
strictly reduces the total distance of these vertices, and the lemma holds. \\[1mm]
(b) Let $H$ and $H'$ be  two distinct $\kappa$-connected path-complete graphs of order $n$. 
It suffices to show that 
\begin{equation} \label{eq:either-H<H' or H'<H}
H' \lneqq H \quad \textrm{or} \quad H \lneqq H', 
\end{equation}
To see this observe that if $H' \lneqq H$, then $m(H') < m(H)$, and by part (a), we have $\rho(H') > \rho(H)$. Similarly, if $H \lneqq H'$, then 
$m(H) < m(H')$, and $\rho(H) > \rho(H')$. In both cases part (a) of the lemma holds.

To prove \eqref{eq:either-H<H' or H'<H}, note that 
there exist $k, k', a, a', b, b' \in \mathbb{N}$ with $a,a' \geq \kappa$ such that 
\[ H = K_1 + [K_{\kappa}]^{\ell} + K_{a} + K_{b}, \quad 
  H' = K_1 + [K_{\kappa}]^{\ell'} + K_{a'} + K_{b'}. \]  
Since $H \neq H'$, it follows that $(\ell,b)\neq (\ell',b')$. We have either  
$(\ell,b) \preceq (\ell',b')$ or $(\ell',b') \preceq (\ell,b)$. 
Without loss of generality we may assume the former. 
   
First assume that $\ell=\ell'$ and $b <b'$. Then $H$ is obtained from $H'$ by 
adding edges between all vertices of a set of $b'-b$ vertices of $K_{b'}$ and 
all vertices in the rightmost $K_{\kappa}$, and so we have $H' \lneqq H$. 
Now assume that $\ell < \ell'$. Then we obtain 
the graph $H$ from
$H'$ by adding edges as follows. By joining all vertices in the rightmost complete
graph $K_{b'}$ to all vertices in the third rightmost complete graph $K_{\kappa}$  
in $H'$ we obtain the graph $K_1 + [K_{\kappa}]^{\ell'} + K_{a'+b'}$. Applying this 
operation again yields the graph $K_1 + [K_{\kappa}]^{\ell'-1} + K_{a'+b'+\kappa}$. 
After repeating this operation a total of $\ell'-k$ times 
we obtain $K_1 + [K_{\kappa}]^{\ell} + K_{\kappa} +  K_{a'+b'+(\ell'-\ell-1)\kappa}$, 
which equals $K_1 + [K_{\kappa}]^{\ell} + K_{\kappa} + K_{a+b-\kappa}$. 
Adding edges between all vertices of a set of $a-\kappa$ vertices of $K_{a+b-\kappa}$ and 
all vertices in the rightmost $K_{\kappa}$ yields the graph 
$K_1 + [K_{\kappa}]^{\ell} + K_{a} + K_{b}$, 
which is $H$. Hence $H' \lneqq H$, and \eqref{eq:either-H<H' or H'<H} follows. 
Thus (b) holds. \\[1mm]
(c) Fix $n$ and $\kappa$. If for $\ell, a,b \in \mathbb{N}$ the graph  
$K_1 + [K_{\kappa}]^{\ell} + K_a + K_b$ has order $n$ and is $\kappa$-connected, then
$n=1 + \ell\kappa + a+b$, and 
$a \geq \kappa$. This implies that $\ell = \frac{n-1-a-b}{\kappa} \leq \frac{n-2-\kappa}{\kappa}$, and $b = n-1-\ell\kappa - a \leq n-1 - (\ell+1)\kappa$. 
With respect to the order $\preceq$, the smallest and largest pairs $(\ell,b)$ satisfying 
these conditions are $(1,1)$ and $(\ell_0, b_0)$, respectively, where 
$\ell_0=\lfloor \frac{n-2-\kappa}{\kappa} \rfloor$ and $b_0=n-1- (\ell_0+1)\kappa$. 
It thus follows as in the proof of \eqref{eq:either-H<H' or H'<H} that the 
$\kappa$-connected path-complete graph  $K_1 + K_{\kappa} + K_{n-\kappa-2} + K_1$,  
arising from the pair $(1,1)$, has maximum size among all $\kappa$-connected path-complete 
graphs of order $n$. Simple calculations show that its size is $\binom{n-1}{2}$. 
The $\kappa$-connected path-complete graph 
$K_1 + [K_{\kappa}]^{\ell_0} + K_{n-1-k\ell_0-b_0} + K_{b_0}$ arising from the pair
$(\ell_0,b_0)$ has minimum size among $\kappa$-connected path-complete 
graphs of order $n$. Its size is 
$\frac{1}{2} \big(n(3\kappa-1) - 2\kappa^2 - \kappa+1 - b(\kappa-b)\big)$.

The proof of part (a) shows that if $m(H) < m(H')$, then $H'$ is obtained from $H$
by adding edges, and the number of edges added is a multiple of $\kappa$. Hence the number
of edges of a $\kappa$-connected path-complete graph of order $n$ is congruent
$\binom{n-1}{2}$ modulo $\kappa$. If $H$ is a $\kappa$-connected path-complete 
graph of order $n$, $H = K_1 + [K_{\kappa}]^{\ell} + K_a + K_b$,
then unless $(\ell,b)=(1,1)$, there exists a $\kappa$-connected path-complete graph 
of order $n$ with exactly $\kappa$ more edges than $H$: 
the graph $K_1 + [K_{\kappa}]^{\ell} + K_{a+1} + K_{b-1}$ (if $b>1$) or the 
graph $K_1 + [K_{\kappa}]^{\ell-1} + K_{\kappa} + K_{a+1}$ (if $\ell>1$ and $b=1$). 
This completes the proof of part (c).
\end{proof}

For given $n,m \in \mathbb{N}$ for which there exists a $\kappa$-connected
path-complete $(n,m)$-graph, we define $PK_{n,m,\kappa}$ to be such a graph
of minimum size. It follows from Lemma \ref{la:size-vs-remoteness-in-kappa-connected-pc}(a) 
that there exists at most one $\kappa$-connected path-complete graph of given order 
and size, so $PK_{n,m,\kappa}$ is well-defined. 


\begin{theorem}\label{th:k-rho}	
(a) Let $G$ be a $\kappa$-connected $(n,m)$-graph  with $m \leq \binom{n-1}{2}$. Then 
\begin{equation} \label{eq:bound-on-rho-for-kappa}
\rho(G) \leq \rho (PK_{n,m,\kappa}).
\end{equation}
(b)  Assume that $m \equiv \binom{n-1}{2} \pmod{\kappa}$ and 
$\frac{1}{2} \big(n(3\kappa-1) - 2\kappa^2 - \kappa+1 - b(\kappa-b)\big) 
\leq m  \leq \binom{n-1}{2}$, where $b$ is the integer in $\{1,2,\ldots,\kappa\}$ with
$b \equiv n-1 \pmod{\kappa}$. 
Then equality in (a) holds only if $G= PK_{n,m,\kappa}$. 
\end{theorem}

\begin{proof}
(a) We first prove that
there exists a $\kappa$-connected path-complete $(n,m)$-graph $G'$ with 
\begin{equation} \label{eq:exists-pc-graph-kappa}
\rho(G) \leq \rho(G'). 
\end{equation}
We may assume that $G$ has  maximum remoteness among 
all $\kappa$-connected $(n,m)$-graphs, and that among all such graphs of maximum remoteness, 
$G$ is one of maximum size. Furthermore, let $v \in V(G)$ with $\overline{\sigma}(v)=\rho(G)$, $d = {\rm ecc}_G(v)$, $ N_{i}=\{z \in V(G)|d_{G}(v,z)=i \} $ and $ |N_{i}|=n_i $ for $i \in \{0,1,\ldots,d\}$. Clearly $ n_0=1 $, each $ n_i $ is a positive integer for $ i=1,2,\ldots, d $ and $ \sum_{i=0}^{d} n_i = n $. \\

\noindent {\sc Claim 1:} $G = K_{n_0} + K_{n_1} + \ldots + K_{n_{d}}$. \label{claim1} \\

 \noindent Recall that $G$ has maximum size among all graphs of size at least $m$ for which $\sigma(v,G)$ is maximized. Hence $N_i$ induces a complete subgraph of $G$, and every vertex in $N_{i-1}$ is adjacent to every vertex in $N_i$ for all $i\in \{1,2,\ldots, d\}$, otherwise we could add an edge and increase the size of $G$ without changing $\sigma(v,G)$. Hence Claim 1 follows. \\
 
 
\noindent {\sc Claim 2:} For all $ i \in \{1,2, \ldots, d-3\} $, we have $ n_i = \kappa $.\\
 
\noindent Let $ i \in \{1,2, \ldots, d-3\} $. Since removing the $n_j$ vertices of $N_j$ 
from $G$ yields a disconnected graph, and since $G$ is $\kappa$-connected, we have 
$n_i \geq \kappa$.

Suppose to the contrary that there exists $ j \in \{1,\ldots,d-3\} $ such that 
$ n_{j}>\kappa $. Let $j$ be the 
smallest such value. Then $n_{i} = \kappa$ for all $ i \in \{1,\ldots,j-1\} $. Now consider 
the graph 
$G^{*}=K_{n_0} + K_{n_1} + \ldots + K_{n_j-1}+K_{n_{j+1}+1} +\ldots + K_{n_{d-1}} + K_{n_d}$. 
Then $m(G^*) = m(G) + n_{j+2} - n_{j-1} \geq m(G)$ since $n_{j+2} \geq \kappa = n_{j-1}$ 
and $G^{*}$ is a $\kappa$-connected $(n,m)$-graph such that $\sigma(v,G^{*}) > \sigma(v,G)$, 
and thus $\rho(G^{*}) > \rho(G)$. This contradiction to the maximility of $\rho(G)$
proves Claim 2.  \\

\noindent {\sc Claim 3:} $n_{d-2}=\kappa$. \\

\noindent
Suppose to the contrary that $n_{d-2} \neq\kappa$. As in Claim 2 we get $n_{d-2} > \kappa$. 
Consider the graph 
$G^{*}=K_{n_0} + K_{n_1} + \ldots+ K_{n_{d-2}-1} +K_{n_{d-1}+2}+ K_{n_d-1}$. 
It is easy to verify that $m(G^*)= m(G) + n_d -n_{d-3} + n_{d-2} -1$, and so, 
since $n_{d-3} = \kappa$ and $n_{d-2} \geq \kappa+1$, we have 
$m(G^*) \geq m(G) + n_d > m(G)$. Also $\rho(G^*) = \rho(G)$. This contradicts our choice of $G$ 
as having maximum size among graphs of maximum remoteness, and so Claim 3 follows. \\

\noindent It follows from Claims 1 to 3 that $G$ is a $\kappa$-connected path-complete 
graph. Letting $G'=G$ proves \eqref{eq:exists-pc-graph-kappa}. \\[1mm]

By \eqref{eq:exists-pc-graph-kappa}, there exists a $\kappa$-connected path-complete graph
$G'$ of order $n$ and size at least $m$ with $\rho(G) \leq \rho(G')$. By the 
definition of  $PK_{n,m,\kappa}$ we have $m(G') \geq m(PK_{n,m,\kappa})$. 
By Lemma \ref{la:size-vs-remoteness-in-kappa-connected-pc}(a), we thus have
$\rho(G') \leq \rho(PK_{n,m,\kappa})$. Hence
\[ \rho(G) \leq \rho(G') \leq \rho(PK_{n,m,\kappa}), \]
which proves (a). \\

(b) Now assume that equality holds in \eqref{eq:bound-on-rho-for-kappa}, i.e., that
$\rho(G) = \rho(PK_{n,m,\kappa})$, and furthermore that 
$m \equiv \binom{n-1}{2} \pmod{\kappa}$ and
$\frac{1}{2} [n(3\kappa-1) - 2\kappa^2 - \kappa+1 - b(\kappa-b)] 
\leq m \leq \binom{n-1}{2}$, where $b$ is as defined above. 
It follows from Lemma \ref{la:size-vs-remoteness-in-kappa-connected-pc}(c) that the 
graph $PK_{n,m,\kappa}$ has exactly $m$ edges. 

It follows from part (a) that $G$ has maximum remoteness among all $\kappa$-connected 
$(n,m)$-graphs. We claim that $G$ has maximum size among all such graphs 
maximising the remoteness. Suppose not. Then there exists a $\kappa$-connected 
$(n,m+1)$-graph $G''$ with $\rho(G) = \rho(G'')$. Applying \eqref{eq:bound-on-rho-for-kappa}
to $G''$ we get that
\[ \rho(G) = \rho(G'') \leq \rho(PK_{n,m+1,\kappa}) < \rho(PK_{n,m,\kappa}), \]
where the last inequality follows from 
Lemma \ref{la:size-vs-remoteness-in-kappa-connected-pc} (b) and the fact that 
$m(PK_{n,m,\kappa})= m < m(PK_{n,m+1,\kappa})$. Hence
$G$ has maximum size among all $\kappa$-connected graphs of order $n$ 
maximising the remoteness. 

The proof of (a) shows that, if $G$ has maximum size among all $\kappa$-connected
path-complete $(n,m)$-graphs, then $G$ is a path-complete graph. Hence $G=PK_{n,m',\kappa}$
for some $m'$ with $m' \geq m$. Since by 
Lemma \ref{la:size-vs-remoteness-in-kappa-connected-pc} we have 
$\rho(PK_{n,m',\kappa}) < \rho(PK_{n,m,\kappa})$ if $m'>m$, we have $m'=m$, 
and thus $G=PK_{n,m,\kappa}$, as desired. 
\end{proof}

Evaluating the remoteness of $PK_{n,m,\kappa}$ yields the following corollary.
	
\begin{corollary}   \label{coro:evaluation-of-kappa-connected-path-complete}
Let $G$ be a $\kappa$-connected graph of order $n$ and size $m$, with
$\frac{1}{2} [n(3\kappa-1) - 2\kappa^2 - \kappa+1 - b(\kappa-b)] 
\leq m \leq \binom{n-1}{2}$, where $b$ is the integer in $\{1,2,\ldots, \kappa\}$ with
$b \equiv n-1 \pmod{\kappa}$. Let $m^*$ be the smallest integer with $m^*\geq m$ and 
$m^* \equiv \binom{n-1}{2} \pmod{\kappa}$.
Then 
\[ \rho(G) \leq  \frac{n}{2\kappa} +2 - \frac{1}{\kappa} - \frac{\kappa-1}{n-1} - \frac{m^*}{\kappa(n-1)}, \]
and this bound is sharp.  
\end{corollary}
 
\begin{proof}
Let $m^*$ be as defined above. It follows from Lemma \ref{la:size-vs-remoteness-in-kappa-connected-pc} that
the graph $PK_{n,m,\kappa}$ has size $m^*$. 
Recall that $PK_{n,m, \kappa}=K_{1}+[K_{\kappa}]^{\ell}+K_{a}+K_{b}$ for some $\kappa,\ell,a,b \in \mathbb{N}$.  Then $a+b=n-\ell\kappa-1$.
Let $v$ be the vertex in $K_1$, so $v$ realizes the remoteness. 
Let $H:= PK_{n,m, \kappa}-V(K_a \cup K_b)$. 
Straightforward calculations show that $n(H)=\ell \kappa +1$, 
$\sigma(v,H)=\frac{\kappa}{2}(\ell(\ell+1))$ and 
$m(H)=\ell {\kappa \choose 2}+(\ell-1)\kappa^2+\kappa$. Hence
\[\sigma(v,PK_{n,m, \kappa})=\sigma(v,H)+(\ell+1)(a+b)+b=\frac{\kappa}{2}(\ell(\ell+1))+(\ell+1)(a+b)+b,\]
and, since $m(PK_{n,m,\kappa})=m^*$,  
\begin{align*}
m^* =& \; m(H)+{a+b+\kappa \choose 2}-b\kappa - {\kappa\choose 2} \\
=& \; \ell {\kappa \choose 2}+(\ell-1)\kappa^2+\kappa+{a+b+\kappa \choose 2}-b\kappa - {\kappa\choose 2}+x\\
=& \; \dfrac{\kappa(\ell-1)(3\kappa-1)}{2}+\kappa+{n-(\ell-1)\kappa+1 \choose 2}-b\kappa.
\end{align*}
Define $\epsilon=\rho(PK_{n,m, \kappa})-\left(\dfrac{n}{2\kappa} + 2 -\dfrac{1}{\kappa}-\dfrac{\kappa-1}{n-1}-\dfrac{
m^*}{\kappa(n-1)}\right)$. 
Substituting  the above terms for $\rho(PK_{n,m,\kappa})$ and $m^*$, it is straightforward 
to verify that $\epsilon=0$. 
This proves that 
$\rho(PK_{n,m, \kappa}) 
     = \left(\dfrac{n}{2\kappa} + 2 
     -\dfrac{1}{\kappa}-\dfrac{\kappa-1}{n-1}-\dfrac{ m^*}{\kappa(n-1)}\right)$,
and the corollary follows.
\end{proof}

For $\kappa=1$, Corollary \ref{coro:evaluation-of-kappa-connected-path-complete} yields
Theorem \ref{EntJacSny1976} and, in addition, characterises the extremal graphs
as follows. Note that for $\kappa=1$ we have $m^*=m$.

\begin{corollary} \label{coro:extremal-graphs-for-Entringer-Jackson}
Let $G$ be a connected graph of order $n$ and size $m$, where $n-1 \leq m \leq \binom{n-1}{2}$. Then 
\[ \rho(G) \leq  \frac{n+2}{2} - \frac{m}{n-1}, \]
with equality if and only if $G=PK_{n,m,1}$. 
\end{corollary}

\section{Maximum remoteness of a $\lambda$-edge-connected graph of given order and size}\label{edge-connectivity}

In this section we determine the maximum remoteness of a $\lambda$-edge connected graph of given order $n$ and size at least $m$. 
For $\lambda=1$, the maximum remoteness was determined in Theorem \ref{EntJacSny1976} and 
Corollary \ref{coro:extremal-graphs-for-Entringer-Jackson}. Our focus here is on graphs with $\lambda=2,3$. 
The proof strategy we employ is similar to that in the previous section. We first define
$\lambda$-edge-connected path-complete graphs, which play a similar role to the
$\kappa$-connected path-complete graphs in the previous section. Since the structure
of $\lambda$-edge-connected path-complete graphs is more varied than that of 
$\kappa$-connected path-complete graphs, the proofs became a little more elaborate.

\begin{definition}\label{deflambda}
Let $\lambda\in\{2,3\}$. A graph $G$ is said to be a $\lambda$-edge-connected path-complete graph if there exist
$k \in \mathbb{N}\cup\{0\}$ and  $a, b \in \mathbb{N}$ with
\[ G = \left\{ \begin{array}{cc}        
  \left[K_{1}+K_{\lambda}\right]^{k}+K_{a}+K_{b} 
  & \textrm{if $k \geq 1$ and $ab \geq \lambda$, or}\\
  \left[K_{1}+K_{\lambda}\right]^{k}+K_{1}+K_{a}+K_{b} 
& \textrm{if $a \geq \lambda$, or}\\ 
\left[K_1 + K_3\right]^k + K_2 + K_a + K_1 
& \textrm{if $\lambda=3$, $k \geq 1$ and $a \geq 3$.}          
 \end{array} \right. \]
 Denote the graphs defined above by
$G_1^{\lambda}(n,k,b)$, $G_2^{\lambda}(n,k,b)$ and $G_3^{\lambda}(n,k)$, respectively, where in each case $n$ is the order of the graph.
\end{definition}


\begin{lemma}  \label{la:3pc-size-vs-remoteness}
Let $\lambda \in \{2,3\}$ and let $H,H'$ be $\lambda$-edge-connected
path-complete graph of order $n$. \\
(a)  $\rho(H+e) < \rho(H)$ for all $e \in E(\overline{H})$. \\[1mm] 
(b) If $H \neq H'$, then either $m(H) < m(H')$ and $\rho(H) \geq \rho(H')$, 
or $m(H) > m(H')$ and $\rho(H) \leq \rho(H')$. \\[1mm]
(c) We have $\frac{5}{3}n-\varepsilon_2(n) \leq m(H) \leq \binom{n}{2}-1$ if $\lambda =2$, 
 where
$\varepsilon_2(n)$ is defined as 
$2$ if $n\equiv 0 \pmod{3}$, as
$\frac{5}{3}$ if $n\equiv 1 \pmod{3}$, and as
$\frac{1}{3}$ if $n\equiv 2 \pmod{3}$.
We have 
and $\frac{9}{4}n-\varepsilon_3(n) \leq m(H) \leq \binom{n}{2}-1$ if $\lambda =3$,
where $\varepsilon_3(n)$ is defined as 
$3$ if $n\equiv 0 \pmod{4}$, as
$\frac{9}{4}$ if $n\equiv 1 \pmod{4}$, as
$\frac{1}{2}$ if $n\equiv 2 \pmod{4}$, and as
$-\frac{9}{4}$ if $n\equiv 3 \pmod{4}$. \\
For every integer $m$ between the smallest and largest size of a $\lambda$-edge-connected
path-complete graph of order $n$ there exists a $\lambda$-edge-connected
path-complete graph of order $n$ whose size is at least $m$ and at most $m+\lambda-1$. 
\end{lemma}

\begin{proof}
Fix $n$ and $\lambda \in \{2,3\}$. \\
(a) The proof is almost identical to Lemma \ref{la:size-vs-remoteness-in-kappa-connected-pc} and thus omitted. \\[1mm]
(b) Denote by $A_n^{\lambda}$ the set of all $\lambda$-edge-connected path-complete graphs of 
order $n$ that are of the form $G_1^{\lambda}(n,k,b)$  or $G_2^{\lambda}(n,k,b)$. 
By considering the lexicographic ordering of the triples $(k,i,b)$ arising from the graphs $G_i(n,k,b)$ for $i\in \{1,2\}$, we show with
arguments similar to those in the proof of 
Lemma \ref{la:size-vs-remoteness-in-kappa-connected-pc} that 
\begin{equation}
\textrm{If $H,H' \in A_n^{\lambda}$ with $H \neq H'$, then either 
$H \lneqq H'$ or $H' \lneqq H$.} 
\end{equation}
By (a) we conclude that we have either $m(H)<m(H')$ and $\rho(H) > \rho(H')$, or 
$m(H')<m(H)$ and $\rho(H') > \rho(H)$, and so (b) holds if $H, H' \in A_n^{\lambda}$. Since
$A^2_n$ contains all $2$-edge-connected path-complete graphs of order $n$, the
lemma holds for $\lambda=2$.  

To complete the proof of the lemma, it remains to show the statement if $\lambda=3$ 
and at least one of the graphs $H, H'$ is not in $A_n^3$. 
First assume that both, $H$ and $H'$ are not in $A_n^3$. Then $H=G_3^3(n,k)$ and 
$H'=G_3^3(n,k')$ 
for some $k,k' \in \mathbb{N}$. Since $H \neq H'$ we have $k\neq k'$. If $k<k'$, then arguments
similar to those in the proof of Lemma \ref{la:size-vs-remoteness-in-kappa-connected-pc}
show that $H' \lneqq H$ and so $m(H') < m(H)$, a contradiction. 
If $k>k'$, then arguments
similar to those in the proof of Lemma \ref{la:size-vs-remoteness-in-kappa-connected-pc}
show that $H \lneqq H'$, which implies that $\rho(H) > \rho(H')$, and the
lemma holds.

We now assume that exactly one of $H$ or $H'$ is in $A_n^3$, so either $H=G_3^3(n,k)$ or 
$H'=G_3^3(n,k)$ for some $k$. 
We compare $G_3^3(n,k)$ with the $3$-edge-connected path-complete graphs $G_1^3(n,k,n-4k-1)$ and $G_1^3(n,k,n-4k-2)$. Simple calculations show that 
$m(G_1^3(n,k,n-4k-1)) < m(G_3^3(n,k)) <  m(G_1^3(n,k,n-4k-2))$ and that $A_n^3$ contains no
graphs whose size is between $m(G_1^3(n,k,n-4k-1))$ and  $m(G_1^3(n,k,n-4k-2))$. 
It is easy to verify that
$\rho(G_1^3(n,k,n-4k-1)) = \rho(G_3^3(n,k)) > \rho(G_1^3(n,k,n-4k-2))$.

If now $H=G_3^3(n,k)$, then it follows by $m(H) \leq m(H')$ and $H' \in A_n^3$ that 
$m(G_1^3(n,k,n-4k-2)) \leq m(H')$ and $\rho(H) > \rho(G_1^3(n,k,n-4k-2)) \geq \rho(H')$.  
If $H'=G_3^3(n,k)$, then it follows by $m(H) \leq m(H')$ and $H \in A_n^3$ that 
$m(H) \leq m(G_1(n,k,n-4k-1))$ and $\rho(H) \geq \rho(G_1(n,k,n-4k-1)) \geq \rho(H')$.
Part (b) now follows. \\[1mm]
(c) To determine the largest size of a $\lambda$-edge-connected path-complete graph  note 
that the graph $G_2^{\lambda}(n,0,1)$ has size $\binom{n}{2}-1$, 
and the complete graph is not a $\lambda$-edge-connected path-complete graph. 

It is easy to verify that the smallest size among $2$-edge-connected path-complete 
graphs of order $n$ is attained 
by $G_1^2(n,\frac{n-3}{3},2)$ if $n\equiv 0 \pmod{3}$, 
by $G^2_2(n,\frac{n-4}{3},1)$ if $n\equiv 1 \pmod{3}$, and
by $G^2_2(n,\frac{n-5}{3},2)$ if $n\equiv 2 \pmod{3}$. Their sizes are
$\frac{5}{3}n-2$, $\frac{5}{3}n- \frac{5}{3}$ and $\frac{5}{3}n-\frac{1}{3}$ if
$n$ is congruent to $0$, $1$ or $2$, respectively, modulo $3$.

It is easy to verify that the smallest size among $3$-edge-connected path-complete 
graphs of order $n$ is attained 
by $G_1^3(n,\frac{n-4}{4},3)$ if $n\equiv 0 \pmod{4}$, 
by $G_1^3(n,\frac{n-5}{4},1)$ if $n\equiv 1 \pmod{4}$, 
by $G^2_3(n,\frac{n-6}{4},2)$ if $n\equiv 2 \pmod{4}$, and
by $G^2_3(n,\frac{n-7}{4},3)$ if $n\equiv 3 \pmod{4}$. Their sizes are
$\frac{9}{4}n-3$, $\frac{9}{4}n- \frac{9}{4}$, $\frac{9}{4}n- \frac{1}{2}$ 
and $\frac{9}{4}n+\frac{9}{4}$ if
$n$ is congruent to $0$, $1$, $2$, or $3$, respectively, modulo $4$.

Also, for $i\in \{1,2\}$ the difference between the sizes of $G_i^{\lambda} (n,k,b)$ and 
$G^{\lambda}_i(n,k,b-1)$ is at most $\lambda$ if $b>1$, and the difference between
the sizes of $G_2^{\lambda}(n,k,1)$ and $G_1^{\lambda}(n,k, n-k(\lambda+1)-1)$ is $1$.
Hence, the set of sizes of $\lambda$-edge-connected path-complete graphs of order $n$ 
contains no gaps of length $\lambda$ or more. Part (c) follows. 
\end{proof}

The proof of Lemma \ref{la:3pc-size-vs-remoteness} shows that for $\lambda \in \{2,3\}$ 
no two $\lambda$-edge-connected path-complete graphs of order $n$ have the same size. 
Hence, for any $n,m$ for which there exists a non-complete $\lambda$-edge-connected
graph of order $n$ and size $m$, there exists a unique $\lambda$-edge-connected path-complete
graph of order $n$ and size at least $m$ that has minimum size among all such graphs. 
We denote this graph by $PK_{n,m}^{\lambda}$.

\begin{theorem}  \label{theo:3-edge-conn}
Let $\lambda \in \{2,3\}$. 
Let $G$ be a $\lambda$-edge-connected graph of order $n$ and size at least $m$ that is not complete. 
Then 
\[ \rho(G) \leq \rho(PK^{\lambda}_{n,m}). \]
\end{theorem}

\begin{proof}
Let $G$ a $\lambda$-edge-connected graph of order $n$ and size at least $m$. 
We first prove that there exists a $\lambda$-edge-connected path-complete $(n,m)$-graph
$G'$ with
\begin{equation} \label{eq:3-edge-conn-theo-1} 
\rho(G) \leq \rho(G'). 
\end{equation}
We may assume that $G$ is a $\lambda$-edge-connected $(n, m)$-graph such that 
$\rho(G)$ is maximised among all $\lambda$-edge-connected $(n, m)$-graphs. 
We may further assume that, among those graphs, $G$ is one of maximum size.
Let $v$ be a vertex of $G$ with $\overline{\sigma}(v) = \rho(G)$, let 
$d = {\rm ecc}(v,G)$. For $i=0,1,\ldots,d$ let 
$N_i$ be the set of vertices at distance $i$ from $v$, and let $n_i=|N_i|$.
As in the proof of Theorem \ref{th:k-rho}, we have
\[ G = K_{n_0} + K_{n_1} + K_{n_2} + \cdots + K_{n_d}. \]
Clearly, $n_0=1$. Let $j \in \{0,1,\ldots,d\}$. 
If $d\leq 2$, then the statement of the lemma clearly holds, so we assume that $d\geq 3$. 
If the values $n_0, n_1,\ldots,n_j$ alternate between $1$ and $\lambda$, i.e., if for 
$i \in \{0,1,2,\ldots, j\}$ we have $n_i=1$ if $i$ is even, and $n_i=\lambda$ if $i$ is odd,
then we say that $n_0, n_1,\ldots,n_j$ is an initial $(1,\lambda)$-segment.  
We may assume that, among those graphs maximising $\rho$ and $m$, 
our graph $G$ has an initial $(1,\lambda)$-segment of maximum length, i.e., the value
of $j$ is maximum. 

If we have $j \geq d-2$, then $G$ is isomorphic to one of the first two graphs in Definition \ref{deflambda}, and the lemma holds in this case. Hence we assume from now on that 
$j \leq d-3$. 
\\[1mm]
{\sc Claim 1:} $j$ is odd, i.e., $n_j=\lambda$. \\
Suppose to the contrary that $j$ is even, i.e., $n_j=1$.  Since $n_{j+1} \neq \lambda$,
and since $n_jn_{j+1} \geq \lambda$ by the $\lambda$-edge-connectivity of $G$, we have that 
$n_{j+1} \geq \lambda +1$. Then the graph 
\[ G'= K_{n_0} + K_{n_1} + \cdots + K_{n_j} + K_{n_{j+1}-1} + K_{n_{j+2}+1} 
               + K_{n_{j+3}} + K_{n_{j+4}} +  \cdots+  K_{n_d} \]
is a $\lambda$-edge-connected $(n,m)$-graph since 
$m(G') = m(G) + n_{j+3}-n_j = m(G) + n_{j+3}-1 \geq m(G)$, and
has greater remoteness than $G$. This contradiction to
the choice of $G$ proves Claim 1. \\[1mm]
{\sc Claim 2:}  $n_{j+3} =1$. \\
By the maximality of $j$ we have $n_{j+1} \geq 2$. 
We first show that $n_{j+3} \leq \lambda -1$. Suppose not. Then $n_{j+3} \geq \lambda$, 
and the graph 
\[ G'= K_{n_0} + K_{n_1} + \cdots + K_{n_j} + K_{n_{j+1}-1} + K_{n_{j+2}+1} 
               + K_{n_{j+3}} + K_{n_{j+4}} +  \cdots + K_{n_d} \] 
is a $\lambda$-edge-connected $(n,m)$-graph since $m(G') = m(G) +n_{j+3}-n_j \geq m(G)$,
and has greater remoteness than $G$. This contradiction to 
the choice of $G$ proves that $n_{j+3} \leq \lambda-1$. This proves Claim 2 
if $\lambda=2$. \\
To complete the proof of the Claim 2, it suffices to show that 
$n_{j+3} \neq 2$ if $\lambda=3$. Suppose to the contrary that $\lambda=3$ and 
$n_{j+3}=2$. Let 
\begin{align*}
G'=&K_{n_0} + K_{n_1} + \cdots + K_{n_j} + K_{n_{j+1}-1} + K_{n_{j+2}+1} 
       + K_{n_{j+3}} + K_{n_{j+4}} +  \cdots\\
       \quad & + K_{n_{d-2}}+ K_{n_{d-1}+1} + K_{n_d-1}, 
\end{align*}
i.e., $G'$ is obtained from $G$ by moving a vertex from $K_{n_{j+1}}$ to 
$K_{n_{j+2}}$, and moving another vertex from  $K_{n_d}$ to $K_{n_{d-1}}$. 
Clearly, $G'$ is a $3$-edge-connected, $G'$ is an $(n,m)$-graph since 
$m(G')=m(G) +n_{j+3}-n_j + n_{d-2} = m(G) -1 + n_{d-2} \geq m(G)$, and 
$\rho(G') = \rho(G)$.  
Repeating this process until there is only one vertex in $N_{j+1}$, 
we eventually obtain a $3$-edge-connected graph 
with the same order, size and remoteness as $G$, but with a greater
value of $j$.  This contradiction to the maximality of $j$ proves 
that $n_{j+2} \neq 2$, and Claim 2 follows. \\[1mm]
{\sc Claim 3:} $j=d-3$. \\
Recall that $j \leq d-3$. 
Suppose to the contrary that $j\leq d-4$. Then $n_{j+4} >0$.  By Claim 2 and 
since $G$ is $\lambda$-edge-connected, we have $n_{j+2} \geq \lambda$ and 
$n_{j+4} \geq \lambda$. Then the graph 
\[ G' = K_{n_0} + K_{n_1} + \cdots + K_{n_j} + K_{1} + K_{n_{j+2}} 
               + K_{n_{j+3}+n_{j+1}-1} + K_{n_{j+4}} +  \cdots+  K_{n_d} \]
is a $\lambda$-edge-connected $(n,m)$-graph since 
$m(G') = m(G) +(n_{j+1}-1) (n_{j+4}-n_j) \geq m(G)$,
and has greater remoteness than $G$. This contradiction to 
the choice of $G$ proves Claim 3. \\[1mm]
We are now ready to complete the proof of the lemma. By Claim 3 and Claim 2, 
we have $n_d=1$. 
First assume that $n_{d-2} \geq \lambda$. Then the graph 
\[ G'= K_{n_0} + K_{n_1} + \cdots + K_{n_{d-3}} + K_{n_{d-2}-1} + K_{n_{d-1}+2} \]             
i.e., $G'$ is obtained from $G$ by moving a vertex from $K_{n_{d-2}}$ to 
$K_{n_{d-1}}$, and moving another vertex from  $K_{n_d}$ to $K_{n_{d-1}}$. Clearly, $G'$ is a $\lambda$-edge-connected $(n,m)$-graph since $m(G') = m(G)$,
and has the same remoteness than $G$. Since $G'$ is a $\lambda$-edge-connected 
path-complete graph, the lemma holds in this case.  

Now assume that $n_{d-2} \leq \lambda-1$. 
If $\lambda =2$, then we have $n_{d-2}=1$ and thus $j \geq d-2$, a contradiction to 
Claim 3, and the lemma holds in this case. Similarly, if $\lambda=3$ and $n_{d-2}=1$, then 
we obtain $j \geq d-2$, again contradicting Claim 3. 
The only remaining case is $\lambda=3$ and $n_{d-2} = 2$. In this case, $G$ is the 
$3$-edge-connected path-complete graph $G_3^3(n,k)$ for some $k \in \mathbb{N}$. 
This proves \eqref{eq:3-edge-conn-theo-1}.

By \eqref{eq:3-edge-conn-theo-1}, there exists a $\kappa$-connected path-complete graph
$G'$ of order $n$ and size at least $m$ with $\rho(G) \leq \rho(G')$. By the 
definition of  $PK_{n,m}^{\lambda}$ we have $m(G') \geq m(PK_{n,m}^{\lambda})$. 
By Lemma \ref{la:3pc-size-vs-remoteness}(b), we thus have
$\rho(G') \leq \rho(PK_{n,m}^{\lambda})$. Hence
\[ \rho(G) \leq \rho(G') \leq \rho(PK_{n,m}^{\lambda}), \]
which proves the theorem. 

\end{proof}

In the proof of Lemma \ref{la:3pc-size-vs-remoteness}(c), the 
$\lambda$-edge-connected path-complete graphs of minimum size were 
determined. It follows from Theorem \ref{theo:3-edge-conn} and 
Lemma \ref{la:3pc-size-vs-remoteness}(b) that these graphs 
maximise the remoteness among all $\lambda$-edge-connected graphs of order $n$. 
Evaluating the remoteness of these graphs yields the following bound
on remoteness for $2$- and $3$-edge-connected graphs. The bound for
$2$-edge-connected graphs was proved originally by Plesn\'\i k \cite{Ple1984}.

\begin{corollary} \label{coro:bound-in-terms-of-order-and-lambda}
(a) {\rm \cite{Ple1984}} If $G$ is a $2$-edge-connected graph of order $n$, then
\[ \rho(G) \leq \left\{ \begin{array}{cc}
    \frac{n}{3} & \textrm{ if $n\equiv 0 \pmod{3}$ or $n \equiv 1 \pmod{3}$, } \\
    \frac{n}{3} - \frac{2}{3(n-1)} & \textrm{ if $n \equiv 2 \pmod{3}$. } 
    \end{array} \right. \] 
(b) If $G$ is a $3$-edge-connected graph of order $n$, then
\[ \rho(G) \leq \left\{ \begin{array}{cc}
    \frac{n}{4} & \textrm{ if $n\equiv 0 \pmod{4}$ or $n \equiv 1 \pmod{4}$, } \\
    \frac{n}{4} - \frac{1}{2(n-1)} & \textrm{ if $n \equiv 2 \pmod{4}$, }  \\
    \frac{n}{4} - \frac{3}{2(n-1)} & \textrm{ if $n \equiv 3 \pmod{4}$. }     
    \end{array} \right. \]  
The above bounds are sharp.           
\end{corollary}

Since evaluating the exact remoteness of $PK_{n,m}^2$ and $PK_{n,m}^3$ gives a rather 
unpleasant expression, we derive a good approximation below. Recall that the minimum
size of a $\lambda$-edge-connected path-complete graph of order $n$ is 
$\frac{5}{3}n - \varepsilon_2(n)$ for $\lambda=2$ and 
$\frac{9}{4}n -\varepsilon_3(n)$ for $\lambda=3$. For graphs of smaller size, the 
bound in Corollary \ref{coro:bound-in-terms-of-order-and-lambda} can likely not be
improved, even if size is taken into account. 

\begin{proposition}
(a) Let $n,m \in \mathbb{N}$ with $\frac{5}{3}n - \varepsilon_2(n) \leq m < \binom{n}{2}$, 
where $\varepsilon_2$ is as defined in Lemma \ref{la:3pc-size-vs-remoteness}(c). 
Then 
\[ \rho(PK^2_{n,m}) = \frac{n}{3} - \frac{2m}{3(n-1)} + \varepsilon. \]
for some $\varepsilon \in \mathbb{R}$ with
$ \frac{2}{3} <\varepsilon < \frac{5}{3}$. \\
(b) Let $n,m \in \mathbb{N}$ with $\frac{9}{4}n -\varepsilon_3(n) \leq m < \binom{n}{2}$, 
where $\varepsilon_3$ is as defined in Lemma \ref{la:3pc-size-vs-remoteness}(c).
Then 
\[ \rho(PK^3_{n,m}) = \frac{n}{4} - \frac{m}{2(n-1)} + \varepsilon. \]
for some $\varepsilon \in \mathbb{R}$ with
$ 1 <\varepsilon < \frac{3}{2}$. \\
\end{proposition}

{\bf Proof:} 
(a) Recall that $PK^{2}_{n,m}$ equals either $[K_1 + K_2]^k + K_a+K_b$ for some 
$k,a,b \in \mathbb{N}$ or 
$[K_1 + K_2]^k + K_1 + K_a+K_b$ for some $k \in \mathbb{N}\cup \{0\}$, $a,b \in \mathbb{N}$. 
We assume that the size of $PK_{n,m}^2$ is $m$; the case that its size is $m+1$
is almost identical. 
Let $v$ be the vertex that realises the remoteness, i.e., the 
vertex in the leftmost $K_1$, and let  $H := PK_{n,m}-V(K_a \cup K_b)$. Define 
$\varepsilon :=
\rho(PK^2_{n,m}) - \big( \frac{n}{3} - \frac{2m}{3(n-1)} \big)$.
It suffices to show that 
\begin{equation} \label{eq:bound-on-epsilon-for-lambda=2}
\frac{2}{3} < \varepsilon  < \frac{5}{3}.
 \end{equation} 
First consider the case that $PK^2_{n,m} = [K_1 + K_2]^k + K_a+K_b$ where $k \in \mathbb{N}$.  
Then $|V(H)|=3k$. Simple calculations show that  $\sigma_H(v) =3k^2-k$, and 
$m(H) = 5k -2$. Also $a+b= n-3k$. Hence
\[   \sigma(v, PK^2_{n,m})
 = \sigma_H(v) + 2k(a+b)+b 
                    = 3k^2-k +2k(n-3k) + b,
\]
and
\[ m(PK^2_{n,m})  =  m(H) + \binom{a+b+2}{2} -2b-1 
            = 5k-2 + \frac{1}{2} (n-3k+2)(n-3k+1) - 2b-1.
\] 
Since $\rho(PK^2_{n,m}) = \frac{\sigma(v, PK^2_{n,m})}{n-1}$ we have
\begin{eqnarray*} 
\varepsilon & = & \frac{ 3k^2-k +2k(n-3k) + b}{n-1} - \frac{n}{3} \\
    &  &  + \frac{10k-4 + (n-3k+2)(n-3k+1) - 4b-2}{3(n-1)} 
        \\
    & = & \frac{4}{3} - \frac{2k+b}{3(n-1)}.
\end{eqnarray*}      
Since $n=3k+a+b$ we have $0<2k+b < n-1$, and thus $1 < \varepsilon < \frac{4}{3}$.
and \eqref{eq:bound-on-epsilon-for-lambda=2} holds.  

Now consider the case that $PK^2_{n,m} = [K_1 + K_2]^k + K_1 + K_a+K_b$ where 
$k \in \mathbb{N}\cup \{0\}$.  
Then $|V(H)|=3k+1$,  $\sigma_H(v) =3k^2+k$, and 
$m(H) = 5k$. Also $a+b= n-3k-1$. Calculations similar to the above yield that
\[ 
\varepsilon  =  \frac{5}{3} - \frac{9k+b+2}{3(n-1)}.
\]
Since $0< 9k+b+2 <3(n-1)$,  we have $\frac{2}{3} < \varepsilon< \frac{5}{3}$,
and \eqref{eq:bound-on-epsilon-for-lambda=2} follows. \\[1mm]
(b) The proof of (b) is similar to (a) and thus omitted. \hfill $\Box$ 

\begin{corollary}
(a) Let $G$ be a $2$-edge-connected graph of order $n$ and size $m$. 
Then
\[ \rho(G) \leq \left\{ \begin{array}{cc}
   \frac{n}{3} & \textrm{if $m <\lceil \frac{5}{3}n \rceil -2$,} \\
   \frac{n}{3} - \frac{2m}{3(n-1)} +\frac{5}{3} 
         & \textrm{if $m  \geq \lceil \frac{5}{3}n \rceil -2$},
      \end{array} \right. \]
and this bound is sharp apart from an additive constant. \\
(b) Let $G$ be a $3$-edge-connected graph of order $n$ and size $m$. 
Then
\[ \rho(G) \leq \left\{ \begin{array}{cc}
   \frac{n}{4} & \textrm{if $m < \lceil \frac{9}{4}n\rceil -2$,} \\
   \frac{n}{4} - \frac{m}{2(n-1)} + \frac{3}{2} 
        & \textrm{if $m \geq \lceil \frac{9}{4}n\rceil -2$,}
      \end{array} \right. \]
and this bound is sharp apart from an additive constant. 
\end{corollary}

We note that the results in this section cannot easily be extended to values of 
$\lambda$ with $\lambda \geq 4$. For example for $\lambda=4$, graphs of the form
$[K_1 + K_4]^{\ell_1} + K_{a_1}+K_{b_1}$ do not necessarily have maximum remoteness for graphs
of their order and size. Indeed, if such a graph has order $n$ and size $m$, then
their remoteness is $\frac{n}{5} - \frac{2m}{5(n-1)} + O(1)$, while 
a graph of the form $K_1+ K_4 + +[K_2]^{\ell_2} + K_{a_2} + K_{b_2}$ of order $n$ and size $m$
has remoteness $\frac{n}{4} - \frac{m}{2(n-1)} + O(1)$. 
For $m = cn(n-1)$, where $0<c<\frac{1}{2}$ the former graph has remoteness 
$\frac{(1-2c)n}{5}+O(1)$, while the latter has remoteness $\frac{(1-2c)n}{4}+O(1)$.

\section{Maximum remoteness of a triangle-free graph with given order and size}\label{bipartite}

In this section we show that the bounds on the remoteness of graph order $n$ and size $m$ 
in Theorem \ref{EntJacSny1976}  can be improved significantly for triangle-free graphs.

 \begin{theorem}\label{th:rem-bi-G}
Let $ G $ be a connected, triangle-free graph of order $ n $ and size $ m $. Then 
\[ \rho(G) \leq \frac{n}{2} + 2 - \frac{2m}{n-1}. \]
\end{theorem}
\begin{proof}
Let $ G $ be a triangle-free graph of order $ n $ and size $ m $, where
$n-1 \leq m \leq \lfloor \frac{n^2}{4} \rfloor$. 
It suffices to show that, for every $v \in V(G)$,
\begin{equation} \label{eq:bipartite-bound} 
\sigma(v,G) \leq \; \dfrac{(n+4)(n-1)}{2}-2m. 
\end{equation}
We proceed by induction on $m$. For $m=n-1$ the bound in \ref{eq:bipartite-bound} becomes
\[ \sigma(v,G) 
   \leq \frac{(n+4)(n-1)}{2} - 2(n-1) 
   = \frac{n(n-1)}{2}, \]
which holds by Theorem \ref{Zel1968AouHan2011}. 

Now assume that $m>n-1$. Then $G$ is not a tree, that is, $G$ contains a cycle. Let $u$ be an arbitrary vertex on a cycle $C$ such that $d(v,u)$ is minimum among all $u$ lying on a cycle 
($u$ can be $v$). Let $w$ be a vertex of $C$ adjacent to $u$. Consider the graph $G'=G-uw$.
We prove that
\begin{equation} \label{eq:bipartite-1}
d_{G'}(v,w) \geq d_{G}(v,w) +2. 
\end{equation}
Let 
$P_1$ be a shortest $(v,w)$ path in $G'$, and let $P_2: x_0, x_1,\ldots,x_j$ be a shortest
$(v,u)$-path in $G$, where $x_0=v$, $x_j=u$ and $d_G(v,u)=j$. We show that $u$ is on $P_1$. 
Indeed, $x_0$ is on $P_1$, and we let $i$ be the largest value for which $x_i$ is on $P_1$. 
Since $x_i$ is on a cycle, it follows by our choice of $u$ that $x_i=u$, i.e., $i=j$. Hence
$u$ is on $P_1$. 

Now the $(u,w)$-section of $P_1$ has more than two edges, since otherwise, if it had only two
edges, then these two edges together with $uw$ would form a triangle in $G$, a contradiction 
to $G$ being triangle-free. Hence $d_{G'}(v,w) \geq d_{G}(v,u) + 3$.  It is easy to see 
that $d_{G'}(v,u) = d_{G}(v,u)$. Also $d_{G}(v,u) \leq d_G(v,w)-1$, and 
\eqref{eq:bipartite-1} follows.

The graph $G'$ is, connected, triangle-free and has $m-1$ edges. Since 
$d_G(v,x) \leq d_{G'}(v,x)$ for $x \in V(G)-\{w\}$, and $d_{G}(v,w) \leq d_{G'}(v,w)-2$,
we have $\sigma(v,G) \leq \sigma(v,G') - 2$.  
Applying the induction hypothesis to $G'$ we obtain that 
\[\sigma(v,G) \leq \sigma(v,G')-2 \leq \frac{(n+4)(n-1)}{2}-2(m-1) -2 
            = \frac{(n+4)(n-1)}{2} - 2m, \]
as desired. 
\end{proof}

We now show that bound in Theorem \ref{th:rem-bi-G} is sharp, even for bipartite graphs.

\begin{example}
Let $n,m \in \mathbb{N}$ be given and let $t \in \mathbb{N}$ be the smallest value for which
$m-n \leq \left \lfloor \dfrac{t^2}{4} \right \rfloor-t-1$. 
We define $f(t)=\left \lfloor \dfrac{t^2}{4} \right \rfloor -t -(m-n)$. Then clearly  
$1 \leq f(t) \leq \left \lfloor \dfrac{t}{2} \right \rfloor-1$. Define 
$BPK_{n,m}=[K_1]^{n-t+1}+\overline{K}_a+\overline{K}_b+\overline{K}_c$, where 
$a=\left \lfloor \dfrac{t}{2} \right \rfloor-f(t)$, $b=\left \lceil \dfrac{t}{2} \right \rceil-1$ 
and $c=f(t)$. \\
	
Clearly, $BPK_{n,m}$ is a bipartite graph with $n$ vertices and $n-t+\left \lfloor \dfrac{t^2}{4} \right \rfloor-f(t)=m$ edges. Let $v$ be the vertex that realises the remoteness, i.e., the vertex in the leftmost $K_1$.

\begin{align}\label{Eqn1}
\sigma(v,BPK_{n,m})=&\sum_{i=0}^{n-t} i +(n-t+1)\left(\left \lfloor \dfrac{t}{2} \right \rfloor-f(t)\right)+(n-t+2)\left(\left \lceil \dfrac{t}{2} \right \rceil-1\right)\nonumber\\
\quad & \;+(n-t+3)(f(t)).
\end{align}

Simplifying Equation (\ref{Eqn1}) and substituting the value of $f(t)$ we get that 

 \begin{align*}
\sigma(v,BPK_{n,m}) =& \; \dfrac{(n+4)(n-1)}{2}-2m-\dfrac{t^2}{2}+\left\lfloor \dfrac{t}{2} \right\rfloor+2\left\lceil \dfrac{t}{2} \right\rceil-\dfrac{3t}{2}+2\left\lfloor \dfrac{t^2}{4} \right\rfloor \\
 =& \; \dfrac{(n+4)(n-1)}{2}-2m.
  \end{align*}

\noindent Hence $BPK_{n,m}$ attains the upper bounds in Theorem \ref{th:rem-bi-G}. 
\end{example}

\end{document}